\documentclass{article}
\usepackage{maa-monthly}

\theoremstyle{theorem}
\newtheorem{theorem}{Theorem}
\newtheorem{proposition}[theorem]{Proposition}
\newtheorem{lemma}[theorem]{Lemma}
\newtheorem{corollary}[theorem]{Corollary}

\theoremstyle{definition}
\newtheorem{definition}[theorem]{Definition}

\newtheorem{example}[theorem]{Example}

\renewcommand{\epsilon}{\varepsilon}
\renewcommand{\phi}{\varphi}

\renewcommand{\leq}{\leqslant}
\renewcommand{\geq}{\geqslant}
\newcommand{\norm}[1]{\left\lVert#1\right\rVert}
\newcommand{\pair}[1]{\left\langle{#1}\right\rangle}
\newcommand{\abs}[1]{\left\vert#1\right\vert}
\newcommand{\ball}{\operatorname{ball}}
\newcommand{\R}{\mathbb{R}}
\newcommand{\C}{\mathbb{C}}
\newcommand{\N}{\mathbb{N}}
\newcommand{\qtext}[1]{\quad\text{#1}\quad}
\newcommand{\qand}{\qtext{and}}
\newcommand{\qfa}{\qtext{for all}}

\begin{document}

\title{Exploring Convexity in Normed Spaces}
\markright{Convexity}
\author{Ryan L.~Acosta Babb}

\maketitle

\begin{abstract}
    Normed spaces appear to have very little going for them: aside from the hackneyed linear structure,
    you get a norm whose only virtue, aside from separating points, is the Triangle Inequality.
    What could you possibly prove with that? As it turns out, quite a lot.
    In this article we will start by considering basic convexity properties
    of normed spaces, and gradually build up to some of the highlights
    of Functional Analysis, emphasizing how these notions of convexity
    play a key role in proving many surprising and deep results.
\end{abstract}

\noindent
\begin{quotation}\begin{flushright}
\emph{¿D\'onde vas con la desigualdad triangular?}\\
(``Where do you think you're going with just the triangle inequality?'')\\
-- Rafael Pay\'a Albert
\end{flushright}\end{quotation}

\section{The triangle inequality.}
The basic (one might add ``only'') tool in a normed vector space $(X,\norm{\cdot})$ is the \emph{Triangle Inequality}
\begin{equation}\label{eqn:TriangIneq}
    \norm{x+y} \leq \norm{x} + \norm{y}.
\end{equation}

Its meaning should be evident to anyone who has tried crossing a rectangular field ``the long way''
as opposed to cutting diagonally across:
the longest side of a triangle is never greater than the sum of the lengths of the remaining sides.

This may look innocent enough, but as we shall see throughout this article,
the triangle inequality (and some of its cousins) play a crucial role in many types of arguments in functional analysis.
After exploring the geometric consequences of the triangle inequality in some detail,
we shall show how it features in proofs of the uniqueness of Hahn--Banach extensions,
that it can upgrade weak convergence to norm convergence, and even show that a
a mild strengthening of the triangle inequality implies reflexivity---%
a profound notion with far-reaching consequences in the study of Banach Spaces.
But first,
we wish to bring the reader's attention to a different geometric interpretation of the triangle inequality:
\emph{convexity}.

Recall that a set is \emph{convex} if it contains all line segments between any two of its points.
Formally, we say that $C\subset X$ is convex if for any $x,y\in C$ and $t\in[0,1]$, $tx + (1-t)y\in C$.
Inequality \eqref{eqn:TriangIneq} is then \emph{equivalent} to the convexity of the closed unit ball in $X$, that is \[
    \ball{X} := \{x \in X : \norm{x}\leq 1\}.
\]

\begin{proposition}
    A function $p\colon X\to[0,\infty)$ with the property that $p(\lambda x)=\abs{\lambda}p(x)$
    for all $x\in X$ and $\lambda\in\R$, satisfies \eqref{eqn:TriangIneq}
    if, and only if, the set $\{x\in X: p(x)\leq 1\}$ is convex.
\end{proposition}
\begin{proof}
    Let $B_p:=\{x\in X: p(x)\leq 1\}$.
    It is straightforward to check that $B_p$ is convex if $p$ satisfies the triangle inequality.
    Conversely, assume that $B_p$ is convex and consider $x,y\in X$ with $p(x),p(y)\neq 0$.
    If we then let \[
        t := \frac{p(x)}{p(x)+p(y)} \qtext{so} 1-t = \frac{p(y)}{p(x)+p(y)},
    \] then $t\in (0,1)$.
    Since $x/p(x)$ and $y/p(y)$ lie in $B_p$, it follows by convexity that
    $tx/p(x)+(1-t)y/p(y) \in B_p$.
    Unraveling the definitions, this is equivalent to \[
        p\left(\frac{x}{p(x)+p(y)}+\frac{y}{p(x)+p(y)}\right) \leq 1,
    \] and the triangle inequality follows.
\end{proof}

If by ``ball'' we mean the usual (i.e., Euclidean) disk or ball in $\R^2$ and $\R^3$, respectively, then this should be fairly clear.
But what happens at the boundary? Take two points $x$ and $y$ with $\norm{x}_2=\norm{y}_2=1$.
Then it is clear that no point on the line joining $x$ and $y$ will meet the boundary again.
In particular, their midpoint $\tfrac{x+y}{2}$ does not lie on the boundary, i.e., \[
    \norm{\frac{x+y}{2}}_2 < 1.
\] In other words, the triangle inequality is strict:
\begin{equation}\label{eqn:R2strict}
    \norm{x+y}_2 < \norm{x}_2 + \norm{y}_2 \qtext{whenever} x\neq y \text{ on } \partial(\ball{\R^2}).
\end{equation}
Such spaces are called \emph{strictly convex}.

\begin{figure}
    \centering
    \includegraphics{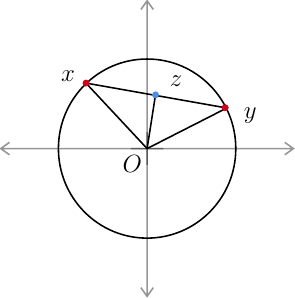}
    \caption{Strict convexity of the Euclidean ball in $\R^2$: the midpoint $z=\tfrac{x+y}{2}$
    never lies on the circle.}
    \label{fig:R2ball}
\end{figure}

\begin{definition}
    A normed space $(X,\norm{\cdot})$ is \emph{strictly convex} if
    whenever $x,y$ lie on the boundary of $\ball{X}$ with $\norm{\tfrac{x+y}{2}}=1$, we have $x=y$.
\end{definition}

Figure \ref{fig:R2ball} suggests a slightly stronger statement.
Let $z = \tfrac{x+y}{2}$. If we separate $x$ and $y$ enough along the circle,
we can make $\norm{z}_2$ (i.e., the segment $Oz$) as small as we like.
In fact, more is true: given any separation $\norm{x-y}_2$, we can ensure a uniform upper bound on $\norm{z}_2$.
This follows from the \emph{Law of Cosines}.

\begin{proposition}\label{prop:R2uniformconvex}
   For any  $\epsilon\in (0,2]$ there is a $\delta\in(0,1)$ such that,
   for all $x,y$ on the unit circle with $\norm{x-y}_2>\epsilon$, we have $\norm{\tfrac{x+y}{2}}_2<\delta$.
\end{proposition}
\begin{proof}
    In Figure \ref{fig:R2ball}, apply the Law of Cosines to the triangles $\triangle Oxy$ and $\triangle Oxz$ to see \[
        \cos(\angle x) = \frac{\norm{x}_2^2+\norm{x-y}_2^2-\norm{y}_2^2}{2\norm{x}_2\norm{x-y}_2}
        = \frac{\norm{x}_2^2+\norm{x-z}_2^2-\norm{z}_2^2}{2\norm{x}_2\norm{x-z}_2},
    \] which, upon noting that $\norm{x}_2=\norm{y}_2=1$ and $x-z=\tfrac{x-y}{2}$, yields \[
        \frac{1}{2}\norm{x-y}_2^2 = 1+\frac{1}{4}\norm{x-y}_2^2-\norm{z}_2^2.
    \] Thus, if $\norm{x-y}_2 > \epsilon$, then $\norm{z}_2<\delta$ provided that $\delta>\sqrt{1-\tfrac{1}{4}\epsilon^2}$.
\end{proof}
A trivial modification to the above argument shows that we can even take $x$ and $y$ to lie anywhere
in the ball, not only on the boundary. This motivates the following definition.
\begin{definition}
    A normed vector space $(X,\norm{\cdot})$ is \emph{uniformly convex} if for every ${\epsilon\in(0,2]}$
    there is a  $\delta\in(0,1)$  such that, whenever $x,y\in\ball{X}$ with $\norm{x-y}>\epsilon$,
    we have $\norm{\tfrac{x+y}{2}}<\delta$.
\end{definition}
Intuitively, the convexity is \emph{uniform} in the sense that given any two points with a prescribed separation,
we have a uniform bound on their midpoint.

Thus, Proposition \ref{prop:R2uniformconvex} tells us that $\R^2$ with the Euclidean norm is uniformly convex.
(In fact, the argument clearly generalizes to $\R^n$ for all $n\geq 2$, and the statement is trivial for $n=1$.)
The next corollary is immediate.
\begin{corollary}\label{cor:uconimplsconv}
    Any uniformly convex space is strictly convex.
\end{corollary}
Whence we derive a rigorous proof of our earlier equation \eqref{eqn:R2strict}.

It should not have escaped the attentive reader that such convexity properties speak more to the virtues of the \emph{norm}
than the underlying vector space.
\begin{figure}[t]
    \centering
    \includegraphics[scale=0.4]{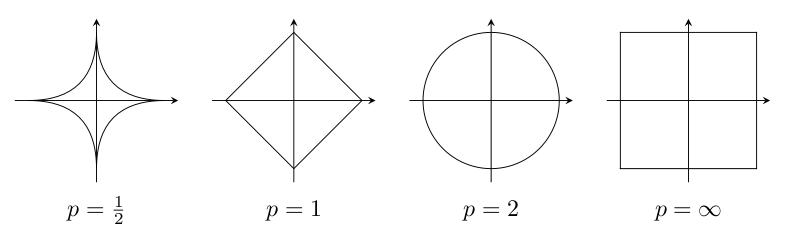}
    \caption{\label{fig:Lpballs}Balls in $(\R^2,\norm{\cdot}_p)$ for various values of $p>0$.
    (Modified from Go Further, CC BY 4.0 $\langle$https://creativecommons.org/licenses/by/4.0$\rangle$, via Wikimedia Commons).}
\end{figure}
Indeed, if we replace the Euclidean norm on $\R^n$ with the $1$-norm or ``Manhattan Norm'' \[
    \norm{x}_1 := \sum_{k=1}^n\abs{x_k},
\] then we lose uniform convexity.
For $n=2$ this is geometrically obvious: the unit ball is now a square (see Figure \ref{fig:Lpballs}),
so that strict convexity fails along the ``straight edges'' of the boundary.
Hence, $(\R^2,\norm{\cdot}_1)$ is not uniformly convex, thanks to Corollary \ref{cor:uconimplsconv}.
We encourage the reader to prove the generalization of this unfortunate fact to $n\geq 2$.

After warming up in finite dimensions, the next ``nice'' example of a normed space is the inner product space.
A moment's reflection on the proof of Proposition \ref{prop:R2uniformconvex} should convince the reader that such spaces are indeed uniformly convex.
\begin{proposition}
    If $(X, \pair{\cdot,\cdot})$ is an inner product space with norm $\norm{x}:=\sqrt{\pair{x,x}}$, then $(X,\norm{\cdot})$ is uniformly,
    hence also strictly convex.
\end{proposition}
The proof is a standard exercise in the use of the Parallelogram Rule, which functions as a substitute for the Law of Cosines.%
\footnote{In fact, one can use the Law of Cosines to \emph{define} angles in a real inner product space as follows:
given linearly independent vectors $x,y$, their \emph{angle} is the unique $\alpha\in(0,\pi)$ such that \[
    \cos\alpha = \frac{\norm{x}^2+\norm{y}^2-\norm{x-y}^2}{2\norm{x}\norm{y}} = \frac{\pair{x,y}}{\norm{x}\norm{y}}.
\]}

Using this fact we can easily devise a norm that is strictly but not uniformly convex.
\begin{example}
    On the vector space $\ell^1$, define the norm \[
        \norm{x} := \norm{x}_{\ell^1}+\norm{x}_{\ell^2}.
    \] If $x$ and $y$ are distinct points on the ball of $(\ell^1,\norm{\cdot})$, then we have \[
        \norm{x+y}_{\ell^2} < \norm{x}_{\ell^2} + \norm{y}_{\ell^2},
    \] as a consequence of Proposition 6. Hence, \[
        \norm{x+y} < \left(\norm{x}_{\ell^1} + \norm{y}_{\ell^1}\right)
        + \left(\norm{x}_{\ell^2} + \norm{y}_{\ell^2}\right) = \norm{x} + \norm{y}  = 2.
    \] Thus, the space is strictly convex. However, if for each $N\in\N$ we define
    two $\ell^1$ sequences $x_N := (x_{N,k})_k$ and $y_N := (y_{N,k})_k$ by \[
        x_{N,k} := \begin{cases}
            1& \text{if } k \leq N,\\
            0& \text{otherwise},
        \end{cases} \qand
        y_{N,k} := \begin{cases}
            1& \text{if } N < k \leq 2N,\\
            0& \text{otherwise},
        \end{cases}
    \] then \[
        \norm{x_N} = \norm{y_N} = N + \sqrt{N} \qand \norm{x_N-y_N} = 2N + \sqrt{2N},
    \] but \[
        \norm{\frac{x_N+y_N}{2}} = \frac{2N}{2} + \frac{\sqrt{2N}}{2};
    \] dividing throughout by $N+\sqrt{N}$ shows that we can make $\tfrac{\norm{x_N+y_N}}{2(N+\sqrt{N})}$
    as close as we like to 1, while keeping $\frac{\norm{x_N-y_N}}{N+\sqrt{N}}\geq \sqrt{2}$.
    Hence, $(\ell^1,\norm{\cdot})$ is not uniformly convex.
\end{example}

Luckily, \emph{most} $\ell^p$---and, indeed, $L^p(\mu)$---spaces are uniformly convex.

\begin{theorem}
    If $1<p<\infty$, then $L^p$ is uniformly convex.
\end{theorem}
\begin{proof}
    See Theorem 4.7.15 in \cite{Bogachev07}.
\end{proof}

It is not hard to construct counterexamples to strict convexity in $L^1$ and $L^\infty$.
\begin{example}\label{ex:ellinfty}
    Consider the following elements in $\ell^\infty$: \[
        x := (1,1,0,\ldots) \qand y := (0,1,1,0,\ldots).
    \] Then, clearly \[
        \norm{x}_{\ell^\infty} = \norm{y}_{\ell^\infty} = 1,
        \qtext{but} \norm{x+y}_{\ell^\infty} = 2.
    \] Hence, $\ell^\infty$ is not strictly convex
    and so, by Corollary 5, it is also not uniformly convex.
\end{example}

\section{Beyond the triangle inequality.}
So what can one do with these various flavors of convexity? As it turns out, quite a lot!

In this section, we will pick three more advanced topics---Hahn--Banach extensions,
weak convergence, and reflexivity---motivating and defining each one as we go.
We will then see how convexity and its cousins interact with these notions
in surprising ways.

First we need to (re)acquaint ourselves with the (topological) dual of a normed vector space.
A linear \emph{functional} on $X$ is just a fancy name for a linear map $f\colon X \to \mathbb{K}$,
where $\mathbb{K}$ is the field of the vector space $X$ (usually $\R$ or $\C$).
Such a map is continuous if, and only if, it is \emph{bounded} in the sense that
the restriction of $f$ to $\ball{X}$ is bounded.\footnote{%
Of course, linear maps cannot literally be bounded since $\abs{f(n x)} = n\abs{f(x)} \to \infty$ as $n\to \infty$!
But this is ``cheating": $f$ is ``really'' only interesting on $\partial(\ball{X})$, i.e., when $\norm{x}=1$,
since $f(x) = \norm{x}f(x/\norm{x})$ ($x\neq0$).}
The set of bounded linear functionals on $X$ is called the \emph{topological dual} or \emph{dual}, for short,
denoted by $X^*$.
(The ``algebraic dual'' is just the set of linear functionals, no continuity required.)
If you haven't ever done so, go ahead and check that $\norm{f}_{X^*} :=\sup \{\abs{f(x)}\colon \norm{x}=1\}$
is a norm on $X^*$---I'll wait.

It is common to write $\pair{f,x}$ for $f(x)$, since this gives us a pleasing symmetry:
we can think of $f\in X^*$ as a function on $X$: $x\mapsto f(x)$;
or we can think of $x\in X$ as a function on $X^*$: $f\mapsto f(x)$ (evaluation at $x$).
We easily check that this is a bounded linear functional, and so belongs to the dual of $X^*$:
the \emph{double dual}, $X^{**}$.

Why should we care about dual spaces? Because they are often much more helpful
than the original vector spaces!

As a warm up, recall from linear algebra that the \emph{only} linear functionals
are the maps $T\colon \mathbb{R}^n\to\R$ of the form $T(v) = u\cdot v$ for some vector $u\in\R^n$.
(This is obvious if you remember that $T$ has to be an $n\times 1$ matrix!)

This is also true in a general inner product space...provided we assume that it is
complete with respect to the norm $\norm{u} = \sqrt{\pair{u,u}}$.
We call a complete inner product space a \emph{Hilbert space},
and it comes with the following powerful theorem.
\begin{theorem}[Riesz Representation Theorem]
    Let $H$ be a Hilbert space with the inner product $\pair{\cdot,\cdot}$,
    and consider a bounded linear functional $T\colon H\to\mathbb{K}$.
    Then, there is a unique vector $u\in H$ such that $T(v)=\pair{u,v}$
    for all $v\in H$. Furthermore, $\norm{T}_{H^*} = \norm{u}$.
\end{theorem}
\begin{proof}
    See Theorem I.3.4.~in \cite{Conway2007}.
\end{proof}

\begin{example}
    Let $f\in C^2[0,1]$ and consider the ODE
    \begin{equation}\label{eqn:ode}
        -u'' + u = f \qtext{with} u(0) = u(1) = 0.
    \end{equation}
    Suppose we have a solution $u\in C^2[0,1]$.
    If we multiply by any $\phi\in C^2[0,1]$ with $\phi(0)=\phi(1)=0$
    and integrate by parts, we obtain
    \begin{equation}\label{eqn:weakinnerp}
        \int_0^1 (u'\phi' + u\phi) = \int_0^1f\phi.
    \end{equation}
    If we can solve \emph{this} equation for some $u$ (and arbitrary $\phi$ as above),
    then we will have solved (\ref{eqn:ode}).
    To do this, define the inner product and norm \[
        \pair{u,\phi}_{\star} := \int_0^1 (u'\phi' + u\phi)
        \qand \norm{\phi}^2_{\star} = \int_0^1\left[(\phi')^2+\phi^2\right],
    \] and consider $H_0^1[0,1]$, the completion of
    \[C_0^2[0,1]:=\{\phi\in C^2[0,1]:\phi(0)=\phi(1)=0\}\]
    with respect to $\norm{\cdot}_{\star}$.
    (Note that $\norm{f}_{\star}$ still makes sense for $f\in C^2[0,1]$,
    since $f$ and $f'$ are continuous and bounded on $[0,1]$.)

    As discussed, the only bounded linear functionals on $H_0^1$ are those of the form
    $\phi\mapsto \pair{u,\phi}_{\star}$ for some $u\in H^1_0[0,1]$.
    What about the functional $\phi\mapsto \int_0^1 f\phi$?
    This is clearly linear, and bounded by the Cauchy--Schwarz inequality:
    \[
        \int_0^1 f\phi \leq \left(\int_0^1 f^2\right)^{1/2}\left(\int_0^1\phi^2\right)^{1/2}
        \leq \norm{f}_{\star}\norm{\phi}_{\star}.
    \] Hence, by the Riesz Representation Theorem there is a $u\in H_0^1[0,1]$ such that \[
        \pair{u, \phi}_{\star} = \int_0^1 f\phi \qfa \phi \in C^2_0[0,1].
    \] Thus, $u$ is a solution to (\ref{eqn:weakinnerp}).
    The extra step of showing that $u\in C^2[0,1]$ actually solves (\ref{eqn:ode}) would take
    us too far afield. (The interested reader may consult chapter 6 of \cite{Evans10}.)
\end{example}

The moral of this example is that working with linear functionals makes our life easier:
by constructing an appropriate functional in the dual space,
we can deduce the existence of a solution to a problem in the original space.

\paragraph{Hahn--Banach Extensions.}
What if we don't have an inner product, but just a norm?
We will still assume completeness, of course; a complete normed space is a \emph{Banach space}.

While we cannot say that \emph{every} functional in $X^*$ arises from an ``inner product''
(because we don't have inner products), we can often construct linear functionals
by defining them on smaller subsets and then extending them to the whole space.

A famous result called the \emph{Hahn--Banach Theorem} implies that we can find an extension
of $f$ to the whole of $X$ without decreasing its norm.
\begin{proposition}\label{prop:HBext}
    Let $X$ be a Banach space with norm $\norm{\cdot}$ and $U\subset X$ be a subspace.
    If $f\colon U\to\mathbb{K}$ is a bounded linear functional on $U$,
    then $f$ admits an extension $g\in X^*$ with the following two properties:
\begin{enumerate}
    \item $g$ agrees with $f$ on $U$: $g|_U \equiv f$; and
    \item $g$ has the same norm as $f$: $\norm{g}_{X^*} = \norm{f}_{U^*}$.
\end{enumerate}

\end{proposition}
A functional $g$ as above is called a \emph{Hahn--Banach extension of $f$}.
(For a proof, as well as a discussion of the Hahn--Banach Theorem and its consequences,
see section III.6 in \cite{Conway2007}.)

Let us see how this works in an example.
\begin{example}\label{ex:ext}
    Let $x\in X$. We seek a functional $g\in X^*$ that tells us what the norm of $x$ is.
    For instance, we could consider the 1-dimensional subspace $Y:= \left\{\alpha x : \alpha\in\mathbb{K}\right\}$
    and the functional $f\colon Y\to\mathbb{K}$ defined as follows: \[
        f(\alpha x) := \alpha\norm{x}.
    \] Now, $f(x) = f(1\cdot x) = \norm{x}$.
    Furthermore, $\abs{f(\alpha x)} = \abs{\alpha}\norm{x} = \norm{\alpha x}$.
    On the other hand, $f(\frac{1}{\norm{x}}\cdot x) = \frac{\norm{x}}{\norm{x}}=1$.
    Hence, $\norm{f}_{Y^*} = 1$.
    But $f\notin X^*$ since $f$ is only defined on $Y$.
    Here is where the Hahn--Banach Theorem comes to the rescue,
    as it guarantees that there is an extension $g\colon X\to\mathbb{K}$
    of $f$ with the same norm, i.e.~$\norm{f}_{Y^*}=\norm{g}_{X^*}$.
\end{example}

In a Hilbert space, such extensions are unique, as can be seen from the Riesz Representation Theorem.
As it turns out, \emph{strict convexity} (hence also uniform convexity) of the dual space
is enough to yield uniqueness in a general Banach space.
\begin{theorem}
    Let $X$ be a Banach space with strictly convex dual $X^*$.
    Suppose $U\subset X$ is a subspace and $f:U\to\mathbb{K}$ a bounded linear functional.
    Then $f$ has a unique Hahn--Banach extension in $X^*$.
\end{theorem}
\begin{proof}
    The existence of an extension follows from the Hahn--Banach Theorem.
    Without loss of generality, we may assume that $\norm{f}=1$.
    So suppose that $g,h\in X^*$ and that both extend $f$ with norm $1$. Then
    \begin{align*}
        \norm{g+h}_{X^*} &= \sup_{x\in\ball{X}}\abs{\pair{g,x}+\pair{h,x}}\\
        &\geq \sup_{x\in \ball{U}}\abs{\pair{g,x}+\pair{h,x}}\\
        &= 2\sup_{x\in \ball{U}}\abs{\pair{f, x}} = 2.
    \end{align*} Thus, $\norm{g+h}=2$, so $g=h$ by the strict convexity of $X^*$.
\end{proof}

\begin{corollary}
    Hahn--Banach extensions are unique in $L^p$ ($1<p<\infty$).
\end{corollary}

This theorem not only grants us uniqueness in $L^p$,
but it also gives us a clue about where to find counter-examples to uniqueness if we relax the assumption
of strict convexity.
\begin{example}
    Recall that $(\ell^1)^* = \ell^\infty$.
    We will exploit Example \ref{ex:ellinfty} to find our extensions.
    Consider the space $U$ to be the span of $e_2=(0,1,0,\ldots)$.
    Define the functional $f\colon \ell^1 \to \C$ by\[
        \pair{f, (x_k)} := x_2,
    \] which is clearly bounded with norm $1$ over $U$.
    (As an element of $\ell^\infty$, $f = e_2$, since $\pair{(f_k),(x_k)} = \sum f_kx_k$,
    and $f$ is picking out the 2nd element of the sequence $(x_k)$.)
    Now take \[
        a := (1,1,0,\ldots) \qand  b := (0,1,1,0,\ldots).
    \] Both are in $\ball\ell^\infty$ and, letting $g$ and $h$ represent the functionals
    induced by $a$ and $b$, we have\[
        \pair{g, x} := \sum_{k}a_kx_k \qand \pair{h, x} := \sum_{k}b_kx_k,
    \] which both agree with $f$ on $U$.
\end{example}

\paragraph{Weak Convergence.}
Another nifty trick used by analysts to find solutions to problems
is to apply some notion of compactness.
\begin{example}[Mini Extreme Value Theorem]\label{exa:dmcv}
    Suppose we have a bounded, continuous function $f\colon[a,b]\to\R$.
    Let $m = \inf_{x\in [a,b]} f(x)$, which exists since $f$ is bounded.
    We want to find a point $x^*\in[a,b]$ where $m$ is attained, i.e., $f(x^*) = m$.
    Well, there is a sequence $x_n\in[a,b]$ such that $f(x_n)\to m$.
    But $[a,b]$ is sequentially compact (or has the Bolzano--Weierstrass property).
    Hence, there is a convergent subsequence $(x_{n_k})_k$ of $(x_n)_n$ with a limit $x^*\in[a,b]$.
    Since $f$ is continuous, it follows that $f(x_{n_k})\to f(x^*)$.
    On the other hand, $f(x_{n_k})\to m$. Hence, $m=f(x^*)$ since limits are unique.
\end{example}
Neat! However, in an infinite-dimensional vector space,
bounded sequences don't always have convergent subsequences.
\begin{example}
    Consider the sequence of sequences $e_n$ with a $1$ in the $n$th place
    and 0s everywhere else.
    In $\ell^2$, this sequence is clearly bounded: $\norm{e_n}_{\ell^2}=1$ for all $n$.
    But whenever $n\neq k$, the vectors $e_n$ and $e_k$ form a right-angle triangle
    in the $(e_n,e_k)$-plane, and so $\norm{e_n-e_k}_{\ell^2} = \sqrt{2}$.
    Thus, $(e_n)$ has no Cauchy subsequences in $\ell^2$, and therefore no convergent subsequences.
\end{example}

Once again, linear functionals come to the rescue.
\begin{definition} We say that a sequence $x_n\in X$ \emph{converges weakly} to $x\in X$ if \[
    \pair{f, x_n} \to \pair{f, x} \qfa f\in X^*.
\] We write $x_n\rightharpoonup x$.
\end{definition}

Oftentimes, weak convergence has nicer properties than norm convergence.
For instance, in $L^p$ spaces ($1<p<\infty$), bounded sequences have \emph{weakly}-convergent subsequences,
and this fact can be used to solve problems using the method form Example \ref{exa:dmcv}.
(This is explored, e.g., in chapter 8 of \cite{Evans10}.)

Clearly, norm convergence implies weak convergence.
The converse is not true in general. (Otherwise there would be little point to weak convergence, wouldn't there?)
However, \emph{sometimes} one can bump up weak convergence to norm convergence.
For instance, it is easy to show that, in a Hilbert space
\begin{equation}\label{eqn:wktostr}
    x_n \rightharpoonup x \text{ and } \norm{x_n} \to \norm{x} \qtext{together imply} x_n \to x.
\end{equation}
(The convergence on the right is norm convergence: $\norm{x_n-x}\to 0$.)
A space with this property is called a \emph{Riesz--Radon} space.
Thus, all Hilbert spaces are Riesz--Radon spaces.
We note that the converse to (\ref{eqn:wktostr}) holds in \emph{any} Hilbert or Banach space.

It turns out that uniformly convex spaces are also Riesz--Radon.
To show this, we record a simple technical lemma.
\begin{lemma}
    Let $(X,\norm{\cdot})$ be a uniformly convex space and $x_n,y_n\in\ball{X}$ such that
    $\norm{\tfrac{x_n+y_n}{2}}\to 1$. Then $\norm{x_n-y_n}\to 0$.
\end{lemma}
\begin{proof}
    Supposing $\norm{x_n-y_n}\not\to 0$, and passing to a subsequence if necessary,
    there is an $\epsilon>0$ such that $\norm{x_n-y_n}\geq \epsilon$ for all $n\geq 1$.
    Since the space is uniformly convex, there is a $\delta\in(0,1)$ such that \[
        1 = \lim_{n\to\infty}\norm{\frac{x_n+y_n}{2}} \leq \delta < 1,
    \] which is patent nonsense.
\end{proof}
\begin{corollary}
    If $x_n\in \ball{X}$, in a uniformly convex space $X$, are such that $\norm{\tfrac{x_n+x_m}{2}}\to 1$ as $n,m\to\infty$,
    then $(x_n)_n$ is Cauchy in $X$.
\end{corollary}
\begin{theorem}
    Every uniformly convex Banach space is a Riesz--Radon space.
\end{theorem}
\begin{proof}
    We may assume that $x\neq 0$ alongside the hypotheses that \[
        x_n \rightharpoonup x \qand \norm{x_n} \to \norm{x}.
    \] Thus (eventually) $x_n\neq 0$ and we may define \[
        y_n := \frac{x_n}{\norm{x_n}} \qand y := \frac{x}{\norm{x}}.
    \] Clearly $y, y_n\in\ball{X}$ and also satisfy \[
        y_n \rightharpoonup y \qand \norm{y_n} \to \norm{y}.
    \]

    By the Hahn--Banach Theorem, there is a functional $f\in X^*$ such that \[
        \norm{f}_{X^*} = 1 \qand \pair{f, y} = 1.
    \] Since $y_n\rightharpoonup y$, it follows that \[
       1  = \pair{f,y} = \lim_{n,m\to\infty} \abs{\pair{f, \tfrac{y_n+y_m}{2}}} \leq \norm{\frac{y_n+y_m}{2}} \leq 1.
    \] Thus, by the previous lemma, $(y_n)$ is Cauchy, so $y_n\to z$ for some $z\in X$.
    But then $y_n\rightharpoonup z$.
    Again by the Hahn--Banach Theorem, weak limits are unique,
    so $y = z$, i.e., $y_n\to y$ in the norm topology of $X$.
    It is now easy to conclude that $x_n\to x$.
\end{proof}

We immediately gain a nice property from Hilbert spaces in all $L^p$ spaces, even if $p\neq 2$.
\begin{corollary}
    Every $L^p$ space ($1<p<\infty$) is a Riesz--Radon space.
\end{corollary}

\paragraph{Reflexivity.}

It's easy to move ``up'' the scale of duality: from $X$ to its dual $X^*$, to \emph{its} dual, $X^{**}$,
and so on.
What is not so clear is how to climb back down.
In other words, given a description of $X^*$, can we deduce what $X$ must have looked like?

Recall that in $\R^n$, every functional in the dual is just an evaluation map,
i.e., $T(v) = \pair{u,v}$ for some $u\in\R^n$.
In general, without an inner product, elements of $X$ cannot be made to act on $X$ itself,
but they \emph{do} act on $X^*$:
for given $x\in X$, define the \emph{evaluation} $\hat{x}\colon X^*\to \mathbb{K}$ via \[
    \hat{x}(f) \equiv \pair{\hat{x},f} := \pair{f,x} \equiv f(x).
\] It is easy to see that such evaluation functionals are bounded, and so we see that \[
    \hat{X} := \left\{\hat{x} : x \in X\right\} \subset X^{**}.
\] In fact, the mapping $\hat{\cdot}\colon X\to X^{**}$ is linear, injective and preserves the norm: \[
    \norm{\hat{x}}_{X^{**}} = \norm{x}.
\] This last point follows from the Hahn--Banach theorem:
given any $x\in X$ there is an $f\in X^*$ with $\norm{f}_{X^*}=1$ and $\pair{f,x} = \norm{x}$.

Now, the million-dollar question is: when does $\hat{X} = X^{**}$, or, when is $x\mapsto\hat{x}$ surjective?

\begin{definition}
    A Banach space $X$ is \emph{reflexive} if $\hat{X} = X^{**}$.
\end{definition}

Clearly all finite-dimensional spaces as well as Hilbert spaces are reflexive,
as are $L^p$ for $1<p<\infty$, since $(L^p)^* = L^q$ when $\tfrac{1}{p}+\tfrac{1}{q}=1$,
so exchanging $p$ and $q$ yields the desired isomorphism.
(This is a little too quick, how can you make it rigorous?)
However, neither $L^1$ nor $L^\infty$ is reflexive.

Since $\hat{X}$ is ``the same as'' (isomorphic to) $X$,
in reflexive spaces we can recover $X$ from $X^{**}$, the dual of $X^{*}$---hurrah!

``Is that it? Just one dual space and then back where we started? Not very exciting...''
(Some people you just cannot please...)
You have a point. Vector spaces are basically just algebra.
To make things interesting (read: ``do analysis'') we need some topology.

Recall (from above) that a sequence $f_n\in X^*$ converges \emph{weakly} to $f\in X^*$ if
\begin{equation}\label{eqn:wkXs}
    \pair{\Phi, f_n} \to \pair{\Phi, f} \qfa \Phi \in X^{**}.
\end{equation}
If $X$ is reflexive, then this becomes \[
    \pair{\hat{x}, f_n} \to \pair{\hat{x}, f} \qfa x \in X,
\] or, in other words,
\begin{equation}\label{eqn:wksXs}
    \pair{f_n, x} \to \pair{f, x} \qfa x \in X.
\end{equation}

In general, $\hat{X}$ will be smaller than $X^{**}$, so (\ref{eqn:wksXs})
will be weaker than (\ref{eqn:wkXs})

\begin{definition}
    The convergence in (\ref{eqn:wksXs}) is called \emph{weak* convergence},
    and its underlying topology, the \emph{weak* topology}.
    (FYI: it's \emph{star}, \emph{weak star}---not ``weak asterisk''!)
\end{definition}

The weak* topology is important because of the following classical theorem.
\begin{theorem}[Banach--Alaoglou]
    The unit ball in $X^*$ is weak*-compact.
\end{theorem}

The next result ties all of this together.
\begin{theorem}
    A Banach space $X$ is reflexive if, and only if, the weak and weak* topologies on $X^*$ coincide.
    Hence, if $X$ is reflexive, then $\ball{X^*}$ is weakly compact.
    Furthermore, if $X$ has a countable dense subset, then the weak topology on $\ball{X}$ is metrizable,
    and so bounded sets in $X^*$ are sequentially compact.
\end{theorem}

\begin{example}(Weak compactness in $L^p(\R^n)$).
    As noted earlier, $L^p(\R^n)$ spaces are reflexive and themselves dual spaces.
    Hence, the ball in $L^p(\R^n)$, i.e., the ball in $(L^q(\R^n))^*$ if $1/p+1/q=1$,
    is weakly compact.
    Hence, bounded sets in $L^p(\R^n)$ are sequentially compact.
    As mentioned above, this is a powerful tool in the analysis of PDE and the Calculus of Variations.
\end{example}

A fuller discussion of reflexivity and weak topologies, as well as proofs of the above theorems
may be found in Chapter V of \cite{Conway2007}.

Returning to convexity, it is a remarkable fact that uniformly convex spaces are automatically reflexive.
(Or perhaps you are beginning to wonder what, if anything, is \emph{not} true in uniformly convex spaces?)
This result is known as the \emph{Milmann--Pettis Theorem},
although the slick and short proof we offer here is due to Ringrose \cite{Ringrose59}.

\begin{theorem}[Milmann--Pettis]
    Any uniformly convex Banach space is reflexive.
\end{theorem}

In the proof below we use the following, equivalent characterization of uniform convexity:
for every $\epsilon\in(0,2]$ there is a $\delta\in(0,1)$ such that
whenever $x,y\in\ball{X}$ with $\norm{x+y}>2-\delta$, we have $\norm{x-y}<\epsilon$.
(That this is equivalent to Definition 4 is a straightforward exercise in notation juggling.)

\begin{proof}
    Suppose $X$ is uniformly convex but not reflexive, so $X^{**}\setminus \hat{X}$ is non-empty.

    Recall (or take on faith) that the image of $\ball{X}$ under $\hat{\cdot}$
    \begin{equation}\label{eqn:u}
        U := \{ \hat{x}\in X^{**} : x \in \ball{X}\}
    \end{equation} is dense in $\ball{X^{**}}$
    with respect to the weak* topology \emph{on} $X^{**}$
    (see Proposition V.4.1 and the surrounding discussion in \cite{Conway2007}).
    It follows that $B := \ball{X^{**}}$ is precisely the weak*-closure of $U$.

    We first note that $U$ is a norm-closed subspace of $B$.
    Indeed, suppose $(\hat{x}_n)_n$ is a Cauchy sequence in $U$.
    Then, $(x_n)_n$ is Cauchy in $\ball{X}$ since \[
        \norm{\hat{x}_n - \hat{x}_m}_{X^{**}} = \norm{x_n - x_m}.
    \] Hence, $x_n\to x$ for some $x\in\ball{X}$, and so, clearly, $\hat{x}_n\to \hat{x}\in U$.

    Now, if $X$ is \emph{not} reflexive, then $U$ is a \emph{proper} closed subset of $B$,
    whence we can find a $\Phi\in B\setminus U$ with unit norm such that
    \begin{equation}\label{eqn:dist}
        \epsilon := \frac{1}{2}d(\Phi, U) > 0.
    \end{equation}

    Since $\norm{\Phi}_{X^{**}}=1$, there is a functional $f\in X^*$ with $\norm{f}_{X^*}=1$
    such that
    \begin{equation}\label{eqn:phi}
        \abs{\Phi(f)-1}< \delta/2,
    \end{equation}
    where $\delta>0$ is chosen in accordance with the uniform convexity of $X$ for our choice of $\epsilon$.
    (This is just a consequence of the definition of the norm on $X^{**}$.)

    Now consider the open set \[
        V := \{\Psi \in X^{**} : \abs{\Psi(f) - 1} < \delta/2 \}.
    \] If $\Psi,\Psi_1\in V\cap U$, then $\abs{\Psi(f)+\Psi_1(f)}>2-\delta$.
    Since $\Psi,\Psi_1\in U$, they are of the form $\hat{x},\hat{x}_1$ for some $x,x_1\in\ball{X}$.
    Thus, \[
        \norm{x+x_1} = \norm{\Psi+\Psi_1}_{X^{**}}\geq 2 -\delta,
    \] and so, by uniform convexity of $X$, \[
        \norm{\Psi-\Psi_1}_{X^{**}} = \norm{x-x_1} < \epsilon.
    \] In other words, for a fixed $\Psi\in V\cap U$, we have \[
        V \cap U \subset \Psi + \epsilon U,
    \] and the latter set is closed, and hence weak*-closed.

    We now show that $\Phi\in \Psi +\epsilon U$.
    To do so, recall that $U$ is weak*-dense in $B$, so we can find a sequence $\hat{x}_n \in U$
    such that $\hat{x}_n \to \Phi$ in the weak* sense.
    But $\Phi \in V$ by (\ref{eqn:phi}), and $V$ is a weak*-open set.
    Hence, eventually, $\hat{x}_n \in V$. This shows that $\Phi$ lies in the weak*-closure of $V\cap U$,
    and therefore in the weak*-closed set $\Psi + \epsilon U$.
    But this means that $\norm{\Phi-\Psi}\leq\epsilon$ for some $\Psi\in U$, contradicting \eqref{eqn:dist}.
\end{proof}

Hence, we have another proof of the reflexivity of $L^p$ spaces ($1<p<\infty$).

As a nice cap to our discussion, we deduce the following result from the reflexivity of $L^p$,
which shows how another basic geometric fact about Hilbert spaces still holds in reflexive spaces.
\begin{corollary}
    Let $Y$ be a closed subspace of $L^p$ ($1<p<\infty$).
    Then, for any $f\notin Y$ there is a unique closest point $g\in Y$, to $f$ i.e., \[
        \norm{f-g}_{L^p} = \inf_{h\in Y}\norm{f-h}_{L^p}.
    \]
\end{corollary}
\begin{proof}
    Since $L^p$ is reflexive, a point $g$ as above exists (see Proposition V.4.7 in \cite{Conway2007}).
    Suppose $h\in Y$ were another closest point, i.e., \[
        \norm{f-g} = \norm{f-h} = d(f,Y).
    \] Since $Y$ is a subspace, $\tfrac{1}{2}(g+h)\in Y$ and so \[
        \norm{f-\frac{g+h}{2}} = \frac{\norm{(f-g)+(f-h)}}{2} < d(f,Y)
    \] by the strict convexity of the norm on $L^p$,
    contradicting the fact that $g\in Y$ is the closest point to $f$.
\end{proof}
\[ * \]
To answer Professor Pay\'a's humorous remark, one can go quite far by thinking about the triangle inequality,
and the long trip offers very pleasant views.

\begin{acknowledgment}{Acknowledgment.}
The author wishes to thank Prof Rafael Pay\'a Albert of the University of Granada for their stimulating conversations over the years,
which have remained a source of inspiration for the author to pursue his studies and research in Mathematics.

He would also like to thank the anonymous reviewers for their many great suggestions
for improving the article into its current form.

This work was supported by the EPSRC studentship EP/V520226/1.
For the purpose of open access, the author has applied a Creative Commons Attribution (CC BY) licence
to any Author Accepted Manuscript version arising from this submission.
\end{acknowledgment}

\begin{biog}
\item[Ryan L.~Acosta Babb] obtained a BSc in Mathematics and Philosophy
with a specialism in Logic and Foundations, and a Master's of Advanced Study in Mathematical Sciences
from the University of Warwick.
Not knowing when to quit, he stubbornly went on to pursue a PhD at the same institution.
His main research interests are harmonic and functional analysis,
especially the $L^p$ convergence of series of eigenfunctions for differential operators.
He is also fascinated by the history and philosophy of mathematics,
having delivered expository talks on constructive mathematics,
the incompleteness theorems and the history of Euclid's parallel postulate.
He remains, to this day, an unrepentant bibliophile, who devotes
much of his time outside research devoted to the study of classics,
and may or may not have spent more time attempting to master
the Greek verbal system than oscillatory integrals.
\begin{affil}
Mathematics Institute, University of Warwick, Coventry, CV47AL, United Kingdom\\
\end{affil}
r.acosta-babb@warwick.ac.uk | \url{https://sites.google.com/view/ryan-acosta/home}
\end{biog}

\bibliography{bibliography  }

\begin{thebibliography}{1}

\bibitem{Bogachev07}
Bogachev VI.
\newblock Measure Theory.
\newblock No. v. 1 in Measure Theory. Springer Berlin Heidelberg; 2007.

\bibitem{Conway2007}
Conway JB.
\newblock A Course in Functional Analysis.
\newblock Graduate Texts in Mathematics. Springer New York; 2007.

\bibitem{Evans10}
Evans LC.
\newblock Partial Differential Equations.
\newblock 2nd ed. Graduate studies in mathematics. American Mathematical
  Society; 2010.

\bibitem{Ringrose59}
Ringrose JR.
\newblock A Note on Uniformly Convex Spaces.
\newblock Journal of the London Mathematical Society. 1959;s1-34(1):92-2.

\end{thebibliography}

\vfill\eject

\end{document}